\documentclass[10pt]{article}
\usepackage[all,2cell]{xy} \UseAllTwocells \SilentMatrices
\usepackage{latexsym,amsfonts,amssymb}
\usepackage{amsmath,amsthm,amscd}
\usepackage{hyperref,psfrag}
\usepackage{color}
\usepackage{etoolbox}
\usepackage[dvips]{epsfig}

\usepackage{graphicx}
\usepackage[centertableaux]{ytableau}
\usepackage{a4wide}
\usepackage{geometry}

\usepackage{epigraph,wrapfig}

\normalfont\upshape

\usepackage{fancyhdr}
\theoremstyle{definition}
\newtheorem{thm}{Theorem}[section]
\newtheorem{cor}[thm]{Corollary}

\newtheorem{lem}[thm]{Lemma}
\newtheorem{rem}[thm]{Remark}
\newtheorem{prop}[thm]{Proposition}
\newtheorem{defn}[thm]{Definition}
\newtheorem{example}[thm]{Example}

\numberwithin{equation}{section}

\usepackage{bbm}

\def\N{{\mathbbm N}}

\def\Z{{\mathbbm Z}}

\def\1{{\mathbbm{1}}}

\newcommand{\Hom}{{\rm Hom}}

\def\lra{{\longrightarrow}}
  
\def\udmod{{\mbox{-}\mathrm{\underline{mod}}}}  

\def\Ext{{\mathrm{Ext}}}

\def\Id{\mathrm{Id}}
\def\mc{\mathcal}

\def\shuffle{\,\raise 1pt\hbox{$\scriptscriptstyle\cup{\mskip
               -4mu}\cup$}\,}

\newcommand{\refequal}[1]{\xy {\ar@{=}^{#1}
(-1,0)*{};(1,0)*{}};
\endxy}

\newcommand{\mH}{\mathrm{H}} 

\newcommand{\dmod}{\mbox{-}\mathrm{mod}}

\title{A Rickard equivalence for hopfological homotopy categories}
\author{You Qi}

\date{\today}

\begin{document}
%

\maketitle

\begin{abstract}
In his paper \cite{Ri}, Rickard presents the stable module category of a self-injective algebra as a Verdier quotient of its derived category by perfect complexes. We present a similar realization of the homotopy category in hopfological algebra as such a Verdier quotient.
\end{abstract}

\setcounter{tocdepth}{2} 


\section{Introduction}
Hopfological algebra, in the sense of \cite{Hopforoots, QYHopf}, was introduced as a generalization of the usual homological algebra. See \cite{QiSussan} for a survey of some recent applications.

Yet one may wonder whether hopfological algebra can be seen more directly through the lens of the usual homological algebra. For instance, one may ask if the construction of homotopy, derived categories in hopfological algebra can be recovered from the usual homotopy and derived categories of (exact) module categories. In this note, we provide a positive answer towards this connection.

More explicitly, let $H$ be a finite-dimensional Hopf algebra, and $B$ be a right $H$-comodule algebra. The abelian category of $B$-modules, denoted $B\dmod$, affords a categorical action by $H\dmod$ via the $H$-comodule algebra structure of $B$. For such a $B$, Khovanov defined in \cite{Hopforoots} the \emph{hopfological homotopy category} $\mc{C}(B,H)$. This background information and some basic facts are presented in Section \ref{sec-basic-def}. 

Our first main goal, established in Section \ref{sec-Frobenius-exact}, is to fit Khovanov's construction into more traditional homological algebra framework. In Lemma \ref{lem-exact-structure}, we show that the category $B\dmod$ is endowed with an \emph{exact structure} $\mc{E}$ in the sense of Quillen, which usually contains ``fewer'' short exact sequences than those in $B\dmod$. Furthermore, this exact category structure is \emph{Frobenius}, and thus allows one to reconstruct (Theorem \ref{thm-reconstruction}) $\mc{C}(B,H)$ as the associated \emph{stable category} in the sense of Keller \cite[Section 2.2]{Ke1}. 

In Section \ref{sec-Rickard}, we give another construction  of $\mc{C}(B,H)$ in the spirit of Rickard's Theorem \cite{Ri} for self-injective algebras. Denote by $\mc{D}^b(B,\mc{E})$ the usual bounded derived category of the exact category $B\dmod,\mc{E})$. Similar as for the abelian category structure on $B\dmod$, the new Frobenius exact structure $\mc{E}$ on $B\dmod$ gives rise to the notion of \emph{$\mc{E}$-perfect complexes} consisting of bounded complexes of $\mc{E}$-projective-injective modules. The class of $\mc{E}$-perfect complexes, denoted $\mc{P}_\mc{E}$, is thick. This allows us to show that there is a triangulated equivalence (Theorem \ref{thm-Rickard})
\[
\mc{C}(B,H)\cong \dfrac{\mc{D}^b(B,\mc{E})}{\mc{P}_{\mc{E}}},
\]
where the right hand side stands for the Verdier quotient of triangulated categories. In particular, when $B=H$, this equivalence agrees with the original theorem of Rickard.
This type of Rickard equivalence adapts to more general Frobenius exact categories and is first proven by Keller-Vossieck \cite{KeVo}. We record another proof here, following the arguments as in \cite{Orlov}.

This work, in part, is motivated by the author's longer term goal to study algebro-topological invariants, such as algebraic K-theory, Hochschild and cyclic homology, in the context of hopfological algebra. Some other approaches to this problem are also suggested in the recent work of Ohara and Tamaki \cite{OhTa}.


\section{Background} \label{sec-basic-def}

\paragraph{Notation.} In this paper, we will let $\Bbbk$ denote a fixed ground field once and for all. Unadorned tensor product ``$\otimes$'' stands for tensor product over $\Bbbk$, and likewise $\Hom$ stands for the space of $\Bbbk$-linear homomorphisms.

Let $H$ be a finite-dimensional Hopf (super-)algebra over $\Bbbk$, whose multiplication is denoted $\Delta$, antipode is denoted $S$ and counit is denoted $\epsilon$. We will usually adopt Sweedler's notation such that the comultiplicaiton on $H$ is written as
\begin{equation}
    \Delta (h) = \sum_h h_1\otimes h_2,
\end{equation}
for any $h\in H$.
Then $H$ is a Frobenius (super-)algebra \cite{LaSw}. It has an element $\Lambda\in H$ called the \emph{left integral}, which is unique up to scaling, satisfying, for any $h\in H$, \begin{equation}
    h \Lambda = \epsilon (h) \Lambda.
\end{equation}

The left integral defines, for any $H$-module $M$, a \emph{canonical $H$-module embedding}
\begin{equation} \label{eqn-canonical-inj}
    \lambda_M: M \lra  M \otimes H , \quad \quad m\mapsto  m \otimes \Lambda.
\end{equation}

We will also use the canonical surjections
\begin{equation} \label{eqn-canonical-surj}
  \rho_M:  M\otimes H \lra M, \quad \quad m\otimes h \mapsto \epsilon(h)m.
\end{equation}

A module of the form $M\otimes H$ is a projective-injective $H$-module, thanks to the following result. Given any $H$-module $M$, let $M_0$ stand for the underlying vector space of $M$ equipped with the trivial $H$-action, i.e., for any $m\in M_0$ and $h\in H$, $hm=\epsilon(h)m$. 

\begin{lem}\label{lem-H-tensor-M}
 For any $M\in H\dmod$, there are isomorphisms of $H$-modules
\begin{equation}
    M\otimes H \cong M_0\otimes H, \quad \quad H\otimes M\cong H\otimes M_0.
\end{equation}
\end{lem}
\begin{proof}
It is an easy exercise to check that
\begin{subequations}
\begin{equation}
    \phi: M\otimes H \lra M_0\otimes H, \quad \quad m\otimes h \mapsto \sum_h S^{-1}(h_1)m\otimes h_2, 
\end{equation}
is a left $H$-module map, whose inverse is given by
\begin{equation}
    \psi: M_0\otimes H \lra M\otimes H, \quad \quad m\otimes h \mapsto \sum_h h_1m \otimes h_2 .
\end{equation}
\end{subequations}
Likewise, the map
\begin{subequations}
\begin{equation}
    \phi^\prime : H\otimes M \lra H\otimes M_0, \quad \quad h\otimes m \mapsto \sum_h h_1\otimes S(h_2)m, 
\end{equation}
and its inverse
\begin{equation}
    \psi^\prime: H\otimes M_0 \lra H\otimes M, \quad \quad h\otimes m \mapsto \sum_h h_1 \otimes h_2m ,
\end{equation}
\end{subequations}
set up the desired isomorphism of left $H$-modules.
\end{proof}

\paragraph{Stable category of Hopf algebras.}
The Frobenius algebra structure on $H$ means that the class of injective $H$-modules coincide with that of projectives. It in turn allows one to form a categorical quotient of $H$-modules by the projective-injective modules.

\begin{defn}
The stable module category $H\udmod$ has the same objects of $H\dmod$, while the morphism space between two objects $U,V$ is given by
\begin{equation}
    \Hom_{H\udmod}(U,V)=\dfrac{\Hom_{H\dmod}(U,V)}{
    \left\{
    f:U\lra V\big|
   {\textrm{ $f$ factors through a } \atop
   \textrm{ projective-injective $H$-module}}
    \right\}
    }
\end{equation}
\end{defn}

\begin{prop}\label{prop-H-udmod-triangulated}
The category $H\udmod$ is triangulated. 
\end{prop}
\begin{proof}
See \cite[Chapter 1]{Hap88}.
\end{proof}

The tensor product $\otimes$ on $H\dmod$ descends onto $H\udmod$ to be an exact bi-functor, which will still be denoted by $\otimes$.

Let us describe, for the sake of completeness, how the triangulated structure is defined on $H\udmod$.
For any $U\in H\udmod$, choose an injective embedding of $U$ in $H\dmod$ 
\begin{equation}
    0 \lra U \stackrel{\iota}{\lra} I_U
\end{equation}
(e.g., take $\iota=\lambda_U$ above) and declare $U[1]:=\mathrm{Coker}(\iota)$.  Likewise, choose a projective covering map
\begin{equation}
    P_U \stackrel{\pi}{\lra} U \lra 0
\end{equation}
(e.g., take $\pi=\epsilon \otimes \Id_U: H\otimes U\lra U$), and declare $U[-1]:=\mathrm{Ker}(\pi)$. The objects $U[1]$, $U[-1]$ are well-defined up to isomorphism and functorial in $U$ in $H\udmod$. The  endofunctors $[1]$, $[-1]$ on $H\udmod$ are inverse of each other.

If $f: U\lra V$ is a map of $H$-modules, then they fit into a commutative diagram
\begin{equation}\label{eqn-triangle-in-H-udmod}
    \begin{gathered}
    \xymatrix{
    0 \ar[r] & U \ar[r]^{\iota} \ar[d]^f & I_U \ar[d] \ar[r] & U[1] \ar[r] \ar@{=}[d]& 0 \\
    0 \ar[r] &  V \ar[r]^g &  C_f \ar[r]^\jmath & U[1] \ar[r] & 0
    }
    \end{gathered} \ ,
\end{equation}
where the left most square is a push-out. The sextuple 
\begin{equation}
    U \stackrel{f}{\lra} V \stackrel{g}{\lra} C_f \stackrel{\jmath}{\lra} U[1]
\end{equation}
is called a \emph{standard distinguished triangle} in $H\udmod$. Any triangle isomorphic to a standard distinguished triangle will be declared a \emph{distinguished triangle}. As part of proof of Proposition \ref{prop-H-udmod-triangulated}, Happel  shows that the class of distinguished triangles satisfies the axioms of triangulated categories.

\paragraph{Hopfological homotopy categories.} \label{sec-hopfo-homotopy}
Let $B$ be a right $H$-comodule algebra. In other words, $B$ is a unital $\Bbbk$-algebra equipped with a map
\begin{equation}
    \Delta_B: B\lra  B \otimes H,
\end{equation}
satisfying
\begin{subequations}
\begin{align}
( \mathrm{Id}_B \otimes \epsilon )\Delta_B = \Id_B, & \quad \quad 
(\Id_B \otimes \Delta )\Delta_B = (\Delta_B \otimes \Id_H)\Delta_B \label{eqn-H-comod-alg-1}
\\
\Delta_B(1)=1_B\otimes 1_H, & \quad \quad \Delta_B(xy)=\Delta_B(x)\Delta_B(y). \label{eqn-H-comod-alg-2}
\end{align}
\end{subequations}
We adapt Sweedler's notation to denote
\begin{equation}
    \Delta_B(b)=\sum_b b_1\otimes b_2
\end{equation}
with it understood that the first components $b_1$'s lie in $B$ and the second $b_2$'s lie in $H$. 

If $U$ is an $H$-module and $M$ is a $B$-module, then $ M\otimes U $ is equipped with a $B$-module structure via $\Delta_B$: for any $b\in B$ and $x \otimes u\in M \otimes U$, 
\begin{equation}
    b\cdot (x\otimes u)= \sum_b (b_1x)\otimes (b_2 u).
\end{equation}
In this way, $B\dmod$ carries a right categorical action by $H\dmod$ given by
\begin{equation}\label{eqn-abelian-H-action}
 B\dmod   \times H\dmod  \lra B\dmod, \quad \quad (M, U)\mapsto M\otimes U.
\end{equation}

\begin{example} Throughout, we will keep in mind the following examples.
\begin{enumerate}
    \item[(1)] Take $B=H$, then $H$ is automatically a comodule algebra over $H$ by taking $\Delta_B= \Delta$, the Hopf algebra comultiplication.
    \item[(2)] Let $A$ be a unital \emph{left $H$-module algebra}, i.e., there is an action map
    \begin{equation}
     \cdot:   H \times A \lra  A, \quad (h,x)\mapsto h\cdot x,
    \end{equation}
    making $A$ into a left $H$-module, such that the following compatibility conditions hold:
    \begin{itemize}
        \item[(i)] For any $x, y \in A$ and $h\in H$, $h\cdot (xy) =\sum_h (h_1\cdot x) (h_2\cdot y)$.
        \item[(ii)] $h\cdot 1_A= \epsilon(h) 1_A$.
    \end{itemize}
    In this case, we can form the \emph{smash product ring} $B:=A \# H$, which is isomorphic to $A\otimes H$ as a vector space, subject to the multiplication rule
    \begin{equation}
        (x\otimes h) (y\otimes k) =\sum_h x(h_1\cdot y) \otimes h_2k,
    \end{equation}
    for any $k,h\in H$ and $x,y\in A$. Set $\Delta_B:=  \Id_A \otimes \Delta$. It is easy to see that $B$ is a right $H$-comodule algebra. 
\end{enumerate}
Note that the first example is a special case of the second by taking $A=\Bbbk$.
\end{example}

\begin{defn}
\begin{enumerate}
    \item[(1)] A morphism $f: M_1\lra M_2$ in $B\dmod$ is called \emph{null-homotopic} if it factors through a $B$-module of the form $ M \otimes H$.  Denote the class of modules whose identity morphism factors through an object of the form $ M \otimes H $ by $\mc{N}_H$
    \item[(2)] The category $\mc{C}(B,H)$ is the quotient of the abelian category $B\dmod$ by the ideal of null-homotopic morphisms. Thus, $\mc{C}(B,H)=B\dmod/\mc{N}_H$.
\end{enumerate}
\end{defn}

The abelian categorical action in \eqref{eqn-abelian-H-action} descends onto a triangulated action.

\begin{prop}\label{prop-H-udmod-action}
The category $\mc{C}(B,H)$ is a triangulated module category over $H\udmod$. The action of $H\udmod$ by exact functors on $\mc{C}(B,H)$ is given by
\begin{equation*}
\mc{C}(B,H) \times H\udmod  \lra \mc{C}(B,H), \quad \quad    (M, U)\mapsto M\otimes U.
\end{equation*}
\end{prop}
\begin{proof}
See \cite[Theorem 1]{Hopforoots}.
\end{proof}

When $B=H$, the category $\mc{C}(B,H)$ coincides with $H\udmod$. Thus this construction is a generalization of the stable module category of $H$.

Similar as for $H\udmod$, the triangulated structure on $\mc{C}(B,H)$ is described as follows. For any $B$-module $M$, define $M[1]:=\mathrm{Coker}(\lambda_M)$, and set $M[-1]$ to be the kernel of the map $\Id_M\otimes \epsilon: M\otimes H \lra M$. Distinguished triangles in $\mc{C}(B,H)$ are constructed as follows. Given a map of $B$-modules $f: M\lra N$, there is a commutative diagram
\begin{equation}\label{eqn-commutative-diagram-in-CBH}
    \begin{gathered}
    \xymatrix{
    0 \ar[r] & M \ar[r]^-{\lambda_M} \ar[d]^f &  M \otimes H \ar[d]^{\psi} \ar[r] & M[1] \ar[r] \ar@{=}[d]& 0 \\
    0 \ar[r] &  N \ar[r]^g &  C_f \ar[r]^\jmath & M[1] \ar[r] & 0
    }
    \end{gathered} \ .
\end{equation}
Then the sextuple 
\begin{equation}\label{eqn-triangle-in-CBH}
    M \stackrel{f}{\lra} N \stackrel{g}{\lra} C_f \stackrel{\jmath}{\lra} M[1]
\end{equation}
is declared a \emph{standard distinguished triangle} in $\mc{C}(B,H)$. Any triangle in $\mc{C}(B,H)$ isomorphic to a standard distinguished triangle is called a \emph{distinguished triangle}.

\begin{lem}\label{lem-zero-obj-criterion}
A $B$-module $M$ descends to the zero object in $\mc{C}(B,H)$ if and only if $M$ is a direct summand of $ M \otimes H$.
\end{lem}
\begin{proof}
By \cite[Lemma 1]{Hopforoots},  a null-homotopic $f:M\lra N$ necessarily factors through the canonical embedding 
$\lambda_M: M\lra  M \otimes H$. Thus, for the identity morphism $\Id_M:M\lra M$ to be null-homotopic, it needs to factor through the canonical embedding $\lambda_M$. This exhibits $M$ as a direct summand of $ M \otimes H $. The converse is clear.
\end{proof}

\begin{cor}\label{cor-proj-B-mod-vanish}
The class of modules $\mc{N}_H$ contains both projective $B$-modules and injective $B$-modules. Consequently, any projective or injective $B$-module descends to the zero object in $\mc{C}(B, H)$.
\end{cor}
\begin{proof}
Given a projective $B$-module $P$, the canonical surjective map $\rho_P: P\otimes H \lra P$ splits since $P$ is projective, exhibiting $P$ as a direct summand of $P\otimes H$. Likewise, given an injective $B$-module $I$, the embedding $\lambda_I: I \lra I\otimes H$ splits and shows that $I\in \mc{N}_H$.
\end{proof}

\paragraph{Short exact sequences and distinguished triangles.}
Next, we point out some relationship between short exact sequences of $B$-modules and distinguished triangles in $\mc{C}(B,H)$. Note that here the situation is more general than that of \cite[Section 4.2]{QYHopf}.

\begin{thm}\label{thm-ses-to-dt}
Let 
$$ 
0 \lra M \stackrel{f}{\lra} N \stackrel{g}{\lra} L \lra 0
$$
be a short exact sequence of $B$-modules that becomes split exact upon tensoring with $H$. Then there is a distinguished triangle\footnote{The connecting map $\mu$ is described explicitly in the proof} in $\mc{C}(B,H)$:
\[
M \stackrel{f}{\lra} N \stackrel{g}{\lra} L \stackrel{\mu }{\lra} M[1].
\]
Conversely, any distinguished triangle in $\mc{C}(B,H)$ is isomorphic to one arising from such a short exact sequence.
\end{thm}
\begin{proof}
Given a short exact sequence as in the statement of the theorem, we have a pushout diagram along $f$, by construction of distinguished triangles,
\begin{equation}
    \begin{gathered}
     \xymatrix{
     0 \ar[r] & M \ar[r]^f \ar[d]_{\lambda_M} & N \ar[r]^g \ar[d]_{\mu_1} & L \ar@{=}[d]\ar[r] & 0\\
     0 \ar[r] &  M \otimes H \ar[r]^-{f_1} & C_f  \ar[r]^{g_1} & L \ar[r] & 0 \\
     }
    \end{gathered} \ .
\end{equation}
At the same time, we also have a commutative diagram
\begin{equation}
    \begin{gathered}
     \xymatrix{
     0 \ar[r] & M \ar[r]^f \ar[d]_{\lambda_M} & N \ar[r]^g \ar[d]_{\lambda_N} & L \ar[d]_{\lambda_L}\ar[r] & 0\\
     0 \ar[r] &  M \otimes H \ar[r]^-{f_2} & N \otimes H \ar@/^/[l]^{s} \ar[r]^{g_2} & L\otimes H \ar[r] & 0
     }
    \end{gathered} \ .
\end{equation}
The universal property of pushout gives us a factorization of the above commutative diagrams:
\begin{equation}\label{eqn-combined-diagram}
    \begin{gathered}
     \xymatrix{
     0 \ar[r] & M \ar[r]^f \ar[d]_{\lambda_M} & N \ar[r]^g \ar[d]_{\mu_1} & L \ar@{=}[d]\ar[r] & 0\\
     0 \ar[r] & M\otimes H \ar@{=}[d]\ar[r]^-{f_1} & C_f \ar[d]_{\mu_2} \ar[r]^{g_1} & L \ar[r] \ar[d]^{\lambda_L}& 0 \\
     0 \ar[r] &  M \otimes H \ar[r]^{f_2}  &  N \otimes H \ar[r]^{g_2} \ar@/^/[l]^s &  L \otimes H \ar[r] & 0 
     }
    \end{gathered} \ .
\end{equation}
Here $\mu_2\mu_1=\lambda_N$ and $\mu_2f_1=f_2$. It follows that
\begin{equation}
    s\mu_2f_1=sf_2=\Id_{M \otimes H },
\end{equation}
and the middle sequence splits as well in \eqref{eqn-combined-diagram} to give $C_f\cong M\otimes H \oplus L$. It follows that, in $\mc{C}(B,H)$, $g_1$ descends to an isomorphism
$ g_1: C_f\cong L$. Set $\mu:=\jmath g_1^{-1}$, where $\jmath:C_f\lra M[1]$ is the map in \eqref{eqn-commutative-diagram-in-CBH}. Since we also have
$ g=g_1\mu_1 $, 
the desired distinguished triangle is exhibited.

The converse part of the result is clear, since we have the short exact sequence of $B$-modules presenting $C_f$ in equation \eqref{eqn-commutative-diagram-in-CBH}:
\begin{equation}\label{eqn-ses-Cf}
    0 \lra M \xrightarrow{(-\lambda_M, f)}  M \otimes H \oplus N \lra C_f \lra 0.
\end{equation}
The sequence splits upon tensoring with $H$ since 
\begin{equation}
    \Lambda \otimes \Id_H: H \lra H\otimes H
\end{equation}
is a split injection of left $H$-modules (see \cite[proof of Lemma 1]{Hopforoots}).
The theorem follows.
\end{proof}

\begin{rem}
The condition of the split exactness of a short exact sequence of $B$-module after tensoring with $H$ is equivalent to requiring that there exists a finite-dimensional $H$-module $V$ such that, after tensoring with $V$, the sequence splits. Indeed, if given a short exact sequence of $B$-modules as above, and $V$ is an $n$-dimensional $H$-module such that
\begin{equation}
0 \lra M\otimes V  \xrightarrow{f\otimes \Id_V} N\otimes V \xrightarrow{g\otimes \Id_V} L\otimes V \lra 0
\end{equation}
splits, then so splits
\begin{equation}
0 \lra M\otimes V \otimes H  \xrightarrow{f\otimes \Id_{V\otimes H}} N\otimes V\otimes H \xrightarrow{g\otimes \Id_{V\otimes H}} L\otimes V \otimes H\lra 0 .
\end{equation}
Then, using Lemma \ref{lem-H-tensor-M}, the sequence is isomorphic to
\begin{equation}
0 \lra (M\otimes H)^{\oplus n}  \xrightarrow{f\otimes \Id_{ H}^{\oplus n} } (N\otimes H)^{^{\oplus n}} \xrightarrow{g\otimes \Id_{ H}^{\oplus n}  } (L\otimes V)^{\oplus n} \lra 0 .
\end{equation}
The desired splitting after tensoring with $H$ follows by restriction to one coordinate. The converse is clear by taking $V=H$.
\end{rem}

\begin{example}\label{eg-H-mod-conflations}
If $B=H$, the homotopy category $\mc{C}(B,H)$ is equivalent to $H\udmod$. Any short exact sequence of $H$-modules leads to a distinguished triangle in $H\udmod$. This well-known fact (see, for instance, \cite[Section 2.7]{Hap88}) can be viewed as a special case of Theorem \ref{thm-ses-to-dt}. This is because any $H$-module, when tensored by $H$, is projective. Thus any short exact sequence of $H$-modules, when tensored with $H$, is split exact.
\end{example}

\begin{example}\label{eg-A-H-mod-conflations}
In \cite{Hopforoots, QYHopf}, one main class of examples is provided by $B= A\# H$ for a left $H$-module algebra $A$. It is shown in \cite[Lemma 4.3]{QYHopf} that, any distinguished triangle in this case arises from a short exact sequence of $B$-module that splits over $A$. If
\begin{equation}
    0 \lra M \stackrel{f}{\lra} N \stackrel{g}{\lra} L \lra 0
\end{equation}
is such a sequence, with $\gamma^\prime: L \lra N$ (resp.~$\beta^\prime: L \lra N$) serving as an $A$-splitting map of $g$, then
\begin{equation}
    0 \lra  M \otimes H \xrightarrow{f\otimes \Id_H }  N \otimes H \xrightarrow{g\otimes\Id_H }  L \otimes H \lra 0
\end{equation}
is $B$-split. Indeed, one can show that (c.f.~\cite[Lemma 4.4]{QYHopf})
\begin{subequations}
\begin{equation}
    \gamma :  L\otimes H \lra  N \otimes H, \quad \quad
    \gamma( x \otimes  h ):= \sum_h (h_2 \cdot \gamma^\prime(S^{-1}(h_1)\cdot x)) \otimes h_3
\end{equation}
\begin{equation}
   \left(\textrm{resp}. \quad N\otimes H \lra  M \otimes H, \quad \quad
    \beta( y \otimes  h ):= \sum_h (h_2 \cdot \beta^\prime(S^{-1}(h_1)\cdot y)) \otimes h_3 \right)
\end{equation} 
\end{subequations}
for any $h\in H$ and $x\in L$ (resp. $y\in N$), is $B$-linear and splits $g\otimes \Id_H$ (resp. $f\otimes \Id_H$) as well.
\end{example}

\section{Exact structure and stable category}\label{sec-Frobenius-exact}
We next show that the $B\dmod$ can be endowed with an \emph{exact category structure} (see, for instance, \cite[Section 4]{KeDerCatUse} and \cite{NeeDerivedExact}) that is \emph{Frobenius}. The stable category of this Frobenius exact structure recovers $\mc{C}(B,H)$.

Recall that an additive category is called \emph{exact} if it is endowed with a distinguished class $\mc{E}$ of sequences called \emph{conflations} (also called admissible short exact sequences)
\begin{equation}\label{eqn-conflation-sequence}
    0 \lra M \stackrel{f}{\lra} N \stackrel{g}{\lra} L \lra 0
\end{equation}
where $f$ is a kernel of $g$ and $g$ is a cokernel of $f$. Such $f$'s are called \emph{inflations} and $g$'s are called \emph{deflations}. The inflations and deflations are confined by the axioms
\begin{enumerate}
    \item[(0)] The identity morphism of $0$ is a deflation.
    \item[(1)] The composition of two deflations is a deflation. The composition of two inflations is an inflation.
    \item[(2)] Deflations are preserved under pullbacks, while inflations are preserved under pushouts.
\end{enumerate}

Motivated by Theorem \ref{thm-ses-to-dt}, we define, on $B\dmod$, a conflation to be a short exact sequence of $B$-modules that becomes split exact after tensoring with $H$. Set $\mc{E}$ to be the class of such conflations.

\begin{lem}\label{lem-exact-structure}
The inflation, deflation and conflation sequences define an exact structure on $B\dmod$.
\end{lem}
\begin{proof}
The facts that the identity morphism of the zero object is a deflation, compositions of two deflations (resp.~inflations) are deflations (resp.~inflations) are easy to check. It thus suffices to verify that deflations are stable under pullbacks and inflations are stable under pushouts. We just check for deflations as an illustration.

If  $g:N\lra L$ is a deflation that fits into a sequence \eqref{eqn-conflation-sequence}, then $g$ splits upon tensoring with $\Id_H$.  Letting $ \mu: L_1 \lra L $ be a map of $B$-modules, we have a pullback diagram
\begin{equation}
    \begin{gathered}
     \xymatrix{
     0 \ar[r] & M \ar[r]^{f_1} \ar@{=}[d] & N_1 \ar[r]^{g_1} \ar[d]^{\mu_1} & L_1 \ar[d]^{\mu}\ar[r] & 0\\
     0 \ar[r] &  M  \ar[r]^-{f} & N  \ar[r]^{g} & L \ar[r] & 0 \\
     }
    \end{gathered} \ .
\end{equation}
Upon tensoring with $H$, the bottom sequence becomes split exact. Let $\gamma: N\otimes H\lra M\otimes H$ be a splitting map for $f\otimes \Id_H$. Then $\gamma (\mu_1\otimes \Id_H)$ splits the top sequence tensored with $H$. 
\end{proof}

\begin{lem}
The class $\mc{E}$ of conflations in $B\dmod$ is stable under tensor products with $H\dmod$.
\end{lem}
\begin{proof}
Given a conflation sequence \eqref{eqn-conflation-sequence} and a module $U\in H\dmod$, we need to show that the sequence tensored with $U\otimes H$ still splits. This follows from the fact that there is a functorial-in-$U$ isomorphism
\begin{equation}
    H\otimes U \lra U\otimes H, \quad \quad h\otimes u\mapsto \sum_h h_1S(h_3)u\otimes h_2,
\end{equation}
which arises from composing the two isomorphisms in Lemma \ref{lem-H-tensor-M}. The inverse map is given by
\begin{equation}
     U \otimes H \lra H \otimes U, \quad \quad u\otimes h\mapsto \sum_h h_2\otimes h_3S^{-1}(h_1) u.
\end{equation}
The result follows.
\end{proof}

Consequently, given $U\in H\dmod$, the functor
\begin{equation}
    B\dmod \lra B\dmod, \quad \quad M\mapsto M\otimes U
\end{equation}
is exact with respect to $\mc{E}$. Furthermore, if $U$ is a projective $H$-module, tensoring with $U$ sends any short exact sequence in $\mc{E}$ into a split short exact sequence. 

\begin{lem}
The category $B\dmod$ with the exact structure $\mc{E}$ is Frobenius. The class of $\mc{E}$-projective-injective objects coincide with $\mc{N}_H$.
\end{lem}
\begin{proof}
Let us show that the class of objects in $\mc{N}_H$ is projective-injective with respect to the specified exact structure $\mc{E}$. It suffices to show this for modules of the form $M\otimes H$. 

Given an inflation $\jmath: N \lra L$ and a $B$-module homomorphism  $f : N\lra M\otimes H$, we have a commutative diagram (solid part):
\begin{equation}
    \begin{gathered}
    \xymatrix{
      N\ar@{=}[d] \ar[rr]^{f} && M\otimes H\ar[d]_{\lambda_{M\otimes H}} \\
     N  \ar[r]^-{\lambda_N} \ar[d]^{\jmath} & N\otimes H \ar[d]_{\jmath\otimes \Id_H}
    \ar[r]^-{f\otimes \Id_H} & M\otimes H\otimes H  \ar@/_2pc/@{-->}[u]_{\gamma_2} \\
      L \ar[r]^-{\lambda_L} & L\otimes H \ar@{-->}@/_2pc/[u]_{\gamma_1}  }
    \end{gathered} \ .
\end{equation}
Since $\jmath\otimes \Id_H$ splits by assumption and $\lambda_{M\otimes H}$ splits because $\lambda_H:H\lra H\otimes H$ is a split injection of $H$-modules, we obtain the (dashed) splitting maps $\gamma_1$ and $\gamma_2$ respectively, and, as a result an extension diagram
\begin{equation}
 \begin{gathered}
    \xymatrix{
    N \ar[rr]^f \ar[dd]_\jmath && M\otimes H\\
    && \\
    L \ar[uurr]_{\gamma_2 (f\otimes \Id_H)\gamma_1\lambda_L}
    }
    \end{gathered} \ .
\end{equation}
Reversing the arrows while replacing the $\lambda$-maps by the corresponding $\rho$-maps shows that objects of the form $M\otimes H$ are also $\mc{E}$-projective. Thus, under this exact structure on $B\dmod$, the class of $\mc{E}$-projectives coincides with the class of $\mc{E}$-injectives. 

Furthermore, given any $M\in B\dmod$, the canonical embedding \eqref{eqn-canonical-inj} and surjection \eqref{eqn-canonical-surj}
\[
\lambda_M : M\lra M\otimes H,\quad \quad
\rho_M: M\otimes H \lra M
\]
are respectively an inflation and a deflation in $\mc{E}$. Thus there are enough $\mc{E}$-projective-injective objects. The last statement is clear, and the result follows.
\end{proof}

As a consequence, Happel's framework \cite[Chapter I]{Hap88} of taking stable quotients of Frobenius categories applies. One obtains the homotopy category $\mc{C}(B,H)$ from $B\dmod$ by modding out the ideal of morphisms that factor through projective-injective objects.  It immediately re-establishes the following result, proved as \cite[Theorem 1]{Hopforoots} and done by explicitly checking the triangulated category axioms.

\begin{thm}\label{thm-reconstruction}
The stable category $\mc{C}(B,H)$ is triangulated. The category $H\udmod$ acts on the right by exact functors on $\mc{C}(B,H)$.
\end{thm}
\begin{proof}
This now directly follows from the above discussion and \cite[Theorem 2.6]{Hap88}.
\end{proof}


\section{A Rickard equivalence}\label{sec-Rickard}
\paragraph{The derived category of an exact category.}
Let us recall the derived category for the exact category $(B\dmod,\mc{E})$. We refer the reader to \cite{KeDerCatUse} and \cite{NeeDerivedExact} for the more general notion of the derived category of an exact category.

As usual, let $\mc{K}^b(B)$ be the abelian category of bounded chain complexes of $B$-modules. An object $M^\bullet\in \mc{K}^b(B)$ consists of bounded complexes 
\begin{equation}
M^\bullet = \left(
\cdots \stackrel{d_{k-2}}{\lra}  M^{k-1}\stackrel{d_{k-1}}{\lra} M^k \stackrel{d_{k}}{\lra} M^{k+1} \stackrel{d_{k+1}}{\lra} \cdots
\right).
\end{equation}
A morphism $f^\bullet: M^\bullet \lra N^\bullet$ is a collection of $B$-module homomorphisms $f^k: M^k \lra N^k$, $k\in \Z$, commuting with the differentials on $M^\bullet$ and $N^\bullet$. A morphism $f^\bullet$ is called \emph{null-homotopic} if there is a collection of maps $h^k: M^k\lra N^{k-1}$ such that
$ f^k = dh^{k}+h^{k+1}d$ holds for all $k\in \Z$. Let $\mc{C}^b(B)$ be the \emph{homotopy category} of $B$ obtained from $\mc{K}^b(B)$ by modding out the ideal of null-homotopic morphisms. In this way, $\mc{C}^b(B)$ becomes a triangulated category.

A complex $M^\bullet$ is called \emph{strictly $\mc{E}$-acyclic} if there are conflations, one for each $k\in \Z$,
\begin{equation}
   0\lra B^{k} \stackrel{\phi_k}{\lra} M^k \stackrel{\psi_k}{\lra} B^{k+1} \lra 0,
\end{equation}
such that $d_k=\phi_{k+1}\psi_k$. A complex $M^\bullet$ is called \emph{$\mc{E}$-acyclic} if it is homotopy equivalent to a strictly $\mc{E}$-acyclic complex.  As is shown in \cite[Lemma 1.1 and 1.2]{NeeDerivedExact}, the full subcategory of $\mc{E}$-acyclic complexes in $\mc{C}^b(B)$ is triangulated and thick. 

\begin{defn}
The \emph{bounded derived category} $\mc{D}^b(B,\mc{E})$ is the Verdier quotient of $\mc{C}^b(B)$ by $\mc{E}$-acyclic complexes.
\end{defn}

By construction, $\mc{D}^b(B,\mc{E})$ is triangulated whose homological shift is denoted $[1]_{\mc{D}}$. Let us also remind the reader how distinguished triangles in $\mc{D}^b(B,\mc{E})$ are constructed (see \cite{KeDerCatUse}). A sequence of maps among chain complexes
\begin{equation}
    0 \lra M^\bullet \stackrel{f^\bullet}{\lra} N^\bullet \stackrel{g^\bullet}{\lra} L^\bullet \lra 0
\end{equation}
is called a \emph{conflation} if, in each degree $k\in \Z$, the sequence
\begin{equation}
    0 \lra M^k \stackrel{f^k}{\lra} N^k \stackrel{g^k}{\lra} L^k \lra 0
\end{equation}
is a conflation in $\mc{E}$. Furthermore, this collection of conflations equips $\mc{K}^b(B)$ with an exact structure.
Such a conflation of chain complexes leads to a distinguished triangle in $\mc{D}^b(B,\mc{E})$. Conversely, any distinguished triangle in $\mc{D}^b(B,\mc{E})$ is isomorphic to one arising in this way.

If $M, N$ are objects of $\mc{D}^b(B,\mc{E})$, we write as in the usual derived category case
\begin{equation}
    \Ext^i_{\mc{D}}(M,N):=\Hom_{\mc{D}^b(B,\mc{E})}(M,N[i]_{\mc{D}}).
\end{equation}

\paragraph{Perfect complexes.}
As in the usual derived category, we consider the notion of perfect complexes.

\begin{defn}
An object in $\mc{D}^b(B,\mc{E})$ is called a \emph{perfect complex} if it is isomorphic to a bounded complex $P^\bullet$ in which each term is $\mc{E}$-projective in $B\dmod$. In other words, each term of $P^\bullet$ lies in $\mc{N}_H$.

We will denote by $\mc{P}_{\mc{E}}$ the full subcategory of perfect complexes in $\mc{D}^b(B,\mc{E})$.
\end{defn}

\begin{lem} \label{lem-thick-PH}
The category $\mc{P}_\mc{E}$ is a thick triangulated subcategory of $\mc{D}^b (B, \mc{E})$.
\end{lem}
\begin{proof}
Denote by $\mc{KP}^b$ the additive category of chain complexes with objects in $\mc{N}_H$, equipped with the termwise (split) exact structure. This is the exact structure $\mc{KP}^b$ inherits from the (termwise) $\mc{E}$-exact structure on $K^b(B)$ discussed above. Clearly, $\mc{KP}^b$ is an idempotent complete exact category.

Passing to the homotopy and derived categories, the category $\mc{KP}^b$ descends to $\mc{P}_\mc{E}\cong \mc{D}^b(\mc{N}_H,\mc{E})$, which embeds fully-faithfully inside $\mc{D}^b(B,\mc{E})$ (see \cite[Theorem 12.1 and Example 12.2]{KeDerCatUse}). The thickness of $\mc{P}_\mc{E}$ is then a consequence of \cite[Theorem 2.8]{BaSch} since the class of projective-injective objects $\mc{N}_H$ is idempotent complete.
\end{proof}

\paragraph{A Rickard type construction.}
Lemma \ref{lem-thick-PH} allows us to form the \emph{Verdier quotient} $\mc{D}^b(B,\mc{E})/\mc{P}_{\mc{E}}$ (see, for instance, \cite[Section 1]{Ri} and the references therein). 

\begin{thm}\label{thm-Rickard}
There is an equivalence of triangulated categories
\begin{equation}
    \mc{R}: \mc{C}(B,H) \lra \mc{D}^b(B,\mc{E})/\mc{P}_{\mc{E}} .
\end{equation}
\end{thm}

With the results of the previous section, the theorem is then a direct consequence of the main result of \cite{KeVo}. However, we record another direct proof following \cite{Orlov}.

\begin{lem}\label{lem-ext-vanishing}
Let $M$ be any $B\dmod$ and $P$ be a module in $\mc{N}_H$. Then, for any $i\geq 1$,
\[
\mathrm{Ext}^i_{\mc{D}}(M,P)=0.
\]
\end{lem}
\begin{proof}
This is true essentially because $P$ is also $\mc{E}$-injective. For the sake of completeness, let us provide an explicit proof, which is modeled on the usual one for the derived category of an abelian category.

A morphism $M\lra P[i]$ in $\mc{D}^b(B,\mc{E})$ is represented by the ``roof'' diagram
\begin{equation}
    \begin{gathered}
     \xymatrix{
    \cdots \ar[r]& Q^{-i-1} \ar[d]\ar[r]^-{d_{-i-1}} & Q^{-i} \ar[r]^-{d_{-i}} \ar[d]^f & Q^{-i+1}\ar[r]^-{d_{-i+1}} \ar[d] & \cdots \ar[r]^{d_{-1}} & Q^0 \ar[r]^s & M \ar[r] & 0\\
     \cdots \ar[r] & 0 \ar[r] & P \ar[r] & 0 \ar[r] & \cdots & & &
     }
    \end{gathered}
\end{equation}
where the top row is acyclic. The commutativity of the diagram implies that we have a conflation
\begin{equation}
    0 \lra B^{-i} \lra Q^{-i} \lra B^{-i+1} \lra 0,
\end{equation}
and $f$ and $d_{-i}$ factor through $B^{-i+1}$ to induce
\begin{equation}
    f^\prime : B^{-i+1} \lra P ,\quad \quad d^\prime: B^{-i+1} \lra Q^{-i+1}.
\end{equation}
Form the pushout along $f^\prime$ and $d^\prime$
\begin{equation}\label{eqn-pushQP}
    \begin{gathered}
    \xymatrix{
0  \ar[r] &     B^{-i+1} \ar[r]^{d^\prime} \ar[d]^{f^\prime} & Q^{-i+1} \ar[d]^{f^{\prime \prime}} \ar[r] & B^{-i+2} \ar[r] \ar@{=}[d] & 0  \\
0 \ar[r] & P \ar[r]^{d^{\prime \prime}} & Z \ar[r] & B^{-i+2} \ar[r] & 0 
}
    \end{gathered} \  .
\end{equation}
Since $P\in \mc{N}_H$ is $\mc{E}$-injective and $d^\prime$ is an inflation, $d^{\prime \prime}$ is also an inflation and must split. Thus we have
\begin{equation}
    Z \cong P \oplus B^{-i+2}.
\end{equation}
Replacing $Q^\bullet \stackrel{s}{\lra} M$ by the equivalent ``roof''
\begin{equation}
\begin{gathered}
     \xymatrix{
    \cdots \ar[r]& 0 \ar[d]\ar[r] & P \ar[r] \ar@{=}[d] & P \oplus B^{-i+2} \ar[r]^-{d_{-i+1}} \ar[d] & \cdots \ar[r]^{d_{-1}} & Q^0 \ar[r]^s & M \ar[r] & 0\\
     \cdots \ar[r] & 0 \ar[r] & P \ar[r] & 0 \ar[r] & \cdots & & &
     }
    \end{gathered} \ ,
    \end{equation}
we see that it is equivalent to the zero roof 
\begin{equation}
\begin{gathered}
     \xymatrix{
    \cdots \ar[r]& 0 \ar[d]\ar[r] & 0 \ar[r] \ar[d] &  B^{-i+2} \ar[r]^-{d_{-i+1}} \ar[d] & \cdots \ar[r]^{d_{-1}} & Q^0 \ar[r]^s & M \ar[r] & 0\\
     \cdots \ar[r] & 0 \ar[r] & P \ar[r] & 0 \ar[r] & \cdots & & &
     }
    \end{gathered} \ .
    \end{equation}
The result follows.    
\end{proof}

Consider the additive functor
\begin{equation}
    \mc{R}^\prime: B\dmod \lra \mc{D}^b(B,\mc{E})/\mc{P}_{\mc{E}},
\end{equation}
which sends a $B$-module $M$ to the object in the quotient category represented by the complex whose only nonzero term is equal to $M$ sitting in cohomological degree zero. By Lemma \ref{lem-zero-obj-criterion}, this functor factors through $\mc{N}_H$ to give us a functor
\begin{equation} \label{eqn-Rickard-functor}
    \mc{R}: \mc{C}(B,H)\lra \mc{D}^b(B,\mc{E})/\mc{P}_{\mc{E}},
\end{equation}
which we will refer to as the \emph{Rickard functor}.

\begin{lem}\label{lem-Rickard-functor-exact}
The Rickard functor is exact.
\end{lem}

\begin{proof}
To show that $\mc{R}$ is an exact functor, let us check that it commutes with the respective shift functors and sends distinguished triangles to distinguished triangles.
To differentiate, let us denote by $[1]_{\mc{C}}$ the homological shift on
$\mc{C}(B,H)$ while using $[1]_{\mc{D}}$ to stand for the homological shift on $\mc{D}^b(B,\mc{E})$ and its quotient
$\mc{D}^b(B,\mc{E})/\mc{P}_{\mc{E}}$.

The conflation sequence
\begin{equation}
    0 \lra M \stackrel{\lambda_M}{ \lra }  M\otimes H \lra \mathrm{Coker}(\lambda_M) \lra 0
\end{equation}
leads to a distinguisehd triangle in $\mc{D}^b(B,\mc{E})$
\begin{equation}
 M \stackrel{\lambda_M}{ \lra } M \otimes H\lra \mathrm{Coker}(\lambda_M) \stackrel{[1]_{\mc{D}}}{\lra} M[1]_{\mc{D}} .
\end{equation}
Since $ M  \otimes H \in \mc{P}_{\mc{E}}$, we have $\mathrm{Coker}(\lambda_M) \cong M{[1]_{\mc{D}}}$ in $\mc{D}^b(B)/\mc{P}_{\mc{E}}$. Thus we have proven
\begin{equation}
    \mc{R}(M[1]_{\mc{C}})=\mc{R}(\mathrm{Coker}(\lambda_M))\cong M[1]_{\mc{D}}.
\end{equation}
Next, given a $B$-module map $f:M\lra N$ there is a short exact sequence of $B$-modules, by equation \eqref{eqn-commutative-diagram-in-CBH},
\begin{equation}\label{eqn-Cf-in-ses}
    0 \lra M \xrightarrow{(\lambda_M, f)} M\otimes H  \oplus N \xrightarrow{\psi+g} C_f \lra 0
\end{equation}
which leads to a distinguised triangle in $\mc{D}^b(B,{\mc{E}})$. Since $ M \otimes H$ descends to zero in the quotient category, we have the resulting distinguished triangle in $\mc{D}^b(B,{\mc{E}})/\mc{P}_{\mc{E}}$
\begin{equation}
    M \stackrel{f}{\lra} N \stackrel{g}{\lra} C_f \stackrel{[1]_{\mc{D}}}{\lra} M[1]_{\mc{D}}.
\end{equation}
This is image of the standard distinguished triangle \eqref{eqn-triangle-in-CBH} under $\mc{R}$. This finishes the proof that $\mc{R}$ is exact.
\end{proof}

\begin{lem}\label{lem-Rickard-functor-surjective}
The Rickard functor is essentially surjective.
\end{lem}
\begin{proof}
We need to show that any object in $\mc{D}^b(B,\mc{E})/\mc{P}_{\mc{E}}$ is isomorphic to an object in the image of $\mc{R}$. By taking an $\mc{E}$-projective resolution, an object $X$ of $\mc{D}^b(B,\mc{E})$  is isomorphic to a bounded-from-above complex of $\mc{E}$-projective $B$-modules
\[
Q^\bullet = \left(
\cdots \lra Q^{r-1}\stackrel{d_{r-1}}{\lra} Q^r \stackrel{d_{r}}{\lra} Q^{r+1} \stackrel{d_{r+1}}{\lra} \cdots \lra Q^s \lra 0 \lra 
\right)
\]
for which $Q^i=0$ if $i>s$ and $Q^\bullet$ is acyclic in degrees less than $  r$. Thus the natural map from $Q^\bullet$ onto the ``stupid'' truncation
\begin{equation}
\sigma_{\leq r}(Q^\bullet):=\left(
\cdots \lra Q^{r-1} \stackrel{d_{r-1}}{\lra} Q^r \lra 0 \lra 0
\right)
\end{equation}
has its cokernel an $\mc{E}$-perfect complex, and thus
is an isomorphism in $\mc{D}^b(B,\mc{E})/\mc{P}_{\mc{E}}$ by Corollary \ref{cor-proj-B-mod-vanish}.
Let $M=\mathrm{Coker}(d_{r-1})$. Then $\sigma_{\leq r}(Q^\bullet) \cong M[r]_{\mc{D}}$ in $\mc{D}^b(B,\mc{E})$. Since $M$ is in the essential image of $\mc{R}$, then so is $M[r]_{\mc{D}}$ since $\mc{R}$ commutes with homological shifts (Lemma \ref{lem-Rickard-functor-exact}). 
\end{proof}

\begin{lem} \label{lem-Rickard-functor-faithful}
The Rickard functor is fully-faithful.
\end{lem}
\begin{proof}
We follow the proof of \cite[Proposition 1.11]{Orlov}.
By Lemma \ref{lem-Rickard-functor-surjective}, it suffices to show that, for any $B$-modules $M$ and $N$, we have an isomorphism 
\begin{equation}
\mc{R}: \Hom_{\mc{C}(B,H)}(M,N)\cong \Hom_{\mc{D}^b(B,\mc{E})/\mc{P}_{\mc{E}}}(\mc{R}(M), \mc{R}(N)).
\end{equation}
By abuse of notation, we will write $M$, $N$ for $\mc{R}(M)$, $\mc{R}(N)$. 

Given a morphism $ M\lra N $ in $\mc{D}^b(B,\mc{E})/\mc{P}_{\mc{E}}$, it is represented by a ``coroof'' in $\mc{D}^b(B,\mc{E})$
\begin{equation}\label{eqn-coroof-1}
\begin{gathered}
 \xymatrix{M  \ar[dr]_g && N\ar[dl]^{t} \\
 & Q^\bullet  &\\
}
\end{gathered}
\ .
\end{equation}
Here the cone of $t$, $C_t^\bullet$, is an object of $\mc{P}_\mc{E}$ by definition of Verdier localization.

Let $P^\bullet$ be an $\mc{E}$-projective resolution of $N$:
\begin{equation}
P^\bullet = \left(
\cdots \stackrel{d_{-2}}{\lra}  P^{-1}\stackrel{d_{-1}}{\lra} P^0 \stackrel{d_{0}}{\lra} P^{1} \stackrel{d_{1}}{\lra} \cdots \lra P^s \lra 0 \lra 
\right).
\end{equation}
For any $k\in \N$, define the ``stupid'' truncation of $P^\bullet$ as
\begin{equation}
\sigma_{\geq -k}P^\bullet = \left(
0 \lra P^{-k } \stackrel{d_{-k}}{\lra}  \cdots \stackrel{d_{-2}}{\lra}  P^{-1}\stackrel{d_{-1}}{\lra} P^0 \stackrel{d_{0}}{\lra} P^{1} \stackrel{d_{1}}{\lra} \cdots \lra P^s \lra 0 \lra 
\right).
\end{equation}
Set $L=\mH^{-k}(\sigma_{\geq -k}P^\bullet)$. We then have a distinguished triangle in $\mc{D}^b(B,\mc{E})$:
\begin{equation}\label{eqn-dt-L-sigma-N}
    L[k] \lra \sigma_{\geq -k}P^\bullet \lra N \stackrel{s}{\lra} L[k+1].
\end{equation}
Form the distinguished triangle
\begin{equation}
  C_t^\bullet [-1] \lra   N \stackrel{t}{\lra} Q^\bullet \lra C_t^\bullet.
\end{equation}
Since $C_t^\bullet $ is a perfect complex, we can find $k$ sufficiently large so that 
\begin{equation}
   \Hom_{\mc{D}}(C_t, L[k])=\Hom_{\mc{D}}(C_t, L[k+1])=0,
\end{equation}
Thus $s: N\lra L[k+1]$ factors through as $s=s_1t$ for some
\begin{equation}
    s_1:   Q^\bullet \lra L[k+1],
\end{equation}
This in turn induces a factorization of the coroof \eqref{eqn-coroof-1} through an equivalent coroof (solid part)
\begin{equation}\label{eqn-coroof-2}
\begin{gathered}
 \xymatrix{M \ar@{-->}[rr]^f  \ar[ddr]_g \ar[dr]^{g_1} && N\ar[ddl]^{t} \ar[dl]_{s} \\
 & L[k+1] &\\
 & Q^\bullet \ar[u]^{s_1}  &\\
}
\end{gathered}
\ ,
\end{equation}
where $g_1:=s_1g$. Applying $\Hom_{\mc{D}^b(B,\mc{E})}(M,\mbox{-}\ )$ to the rotated distinguished triangle \eqref{eqn-dt-L-sigma-N}, and using Lemma \ref{lem-ext-vanishing}, we obtain a surjection
\begin{equation}
  \Hom_{\mc{D}}(M,N) \stackrel{s_*}{\lra}  \Hom_{\mc{D}}(M,L[k+1]) 
  \lra \Ext_{\mc{D}}^1(M,\sigma_{\geq -k} P^\bullet)=0.
\end{equation}
Thus there is necessarily an $f: M \lra N$ (dashed arrow in \eqref{eqn-coroof-2}) such that
$sf=g_1$, which is equivalent to the coroofs \eqref{eqn-coroof-1}, \eqref{eqn-coroof-2}.

If $f=0$ in $\mc{D}^b(B,\mc{E})/\mc{P}_{\mc{E}}$, then we may have chosen $g=0$ with out loss of generality. Then $g_1=s_1g=0$, and $g_1$ factors through a morphism $M\lra \sigma_{\geq - k} P^\bullet$. Again, by Lemma \ref{lem-ext-vanishing}, this morphism necessarily factors through $P^0$, which is ${\mc{E}}$-projective-injective. This finishes the proof of the lemma.
\end{proof}

Now Theorem \ref{thm-Rickard} follows by combining Lemmas \ref{lem-Rickard-functor-exact}, \ref{lem-Rickard-functor-surjective} and \ref{lem-Rickard-functor-faithful}. 

\begin{rem}
In the speical case when $B=H$, Theorem \ref{thm-Rickard} reduces to the special case of Rickard's theorem applied to the Frobenius category $H\dmod$ with the usual abelian exact structure:
\begin{equation}
    H\udmod \cong \mc{D}^b(H)/\mc{P}.
\end{equation}
Here, the conflations in $\mc{E}$ agrees with the usual exact sequences in the abelian category $H\dmod$ (Example \ref{eg-H-mod-conflations}). Note that, this category $\mc{D}^b(H)/\mc{P}$ acts on $\mc{D}^b(B,\mc{E})/\mc{P_\mc{E}}$ by the usual tensor product of chain complexes. Since the Rickard equivalence is set up by sending any object of $H\udmod$ to the one-term complex in $\mc{D}^b(H)/\mc{P}$, the action of these tensor triangulated categories on both sides of Theorem \ref{thm-Rickard} commutes with the Rickard equivalence (c.f.~Proposition \ref{prop-H-udmod-action}).
\end{rem}

\begin{cor}
If $B$ has a finite $H$-equivariant $(B,B)$-bimodule resolution, then $\mc{C}(B,H)=0$.
\end{cor}
\begin{proof}
If $B$ has such a bimodule resolution, then any $M\in B\dmod$ is quasi-isomorphic to a perfect complex.  The result follows from the previous theorem.
\end{proof}

The last result generalizes \cite[Section 1, Example d]{Hopforoots}, as well as the fact that the stable category of $H$ is trivial when $H$ is semisimple (recall that $H$ is either semisimple or of infinite homological dimension since it is Frobenius).

\addcontentsline{toc}{section}{References}


\bibliographystyle{alpha}
\bibliography{qy-bib}

%

\vspace{0.1in}

\noindent Y.~Q.: { \sl \small Department of Mathematics, University of Virginia, Charlottesville, VA 22904, United States} \newline \noindent {\tt \small email: yq2dw@virginia.edu}

%
\end{document}